\documentclass[<opyions>]{elsarticle}
\usepackage{amssymb}
\usepackage{graphicx}
\usepackage{epstopdf}
\usepackage{subfigure}
\usepackage{extarrows}
\usepackage{fancyhdr}
\usepackage{amsthm}
\usepackage{extarrows}
\usepackage{caption}
\usepackage{float}
\usepackage{algorithm}
\usepackage{algorithmic}
\usepackage{multirow}
\usepackage{amsmath}
\usepackage[numbers]{natbib}
\usepackage{graphicx}
\usepackage{amsmath} 
\usepackage{amsfonts} 
\usepackage{mathrsfs} 
\usepackage{booktabs} 
\usepackage{color} 

\usepackage{color}
\usepackage{hyperref}

\usepackage[toc,page,title,titletoc,header]{appendix}
\usepackage{appendix}
\theoremstyle{definition}

\theoremstyle{definition}
\usepackage{hyperref}
\usepackage{latexsym, bm}
\usepackage{fancyhdr}
\usepackage{mathrsfs}
\usepackage{wasysym}
\usepackage{float}
\usepackage{titlesec}
\usepackage{pgfplots}
\usepackage{tikz}
\usepackage{subfigure}
\usepackage{natbib}
\biboptions{numbers,sort&compress}

\usetikzlibrary{arrows,shapes,positioning}
\usetikzlibrary{decorations.markings}
\tikzstyle arrowstyle=[scale=1]
\tikzstyle directed=[postaction={decorate,decoration={markings,
    mark=at position .65 with {\arrow[arrowstyle]{stealth}}}}]
\tikzstyle reverse directed=[postaction={decorate,decoration={markings,
    mark=at position .65 with {\arrowreversed[arrowstyle]{stealth};}}}]

\newtheorem{lemma}{Lemma}[section]
\newtheorem{theorem}{Theorem}[section]

\newtheorem{example}{Example}[section]

\allowdisplaybreaks
\begin{document}
	
\begin{frontmatter}
\title{{\bf Weighted implicit-explicit discontinuous Galerkin methods for two-dimensional Ginzburg-Landau equations on general meshes}}

\author{Zhen Guan\corref{cor1}}
\ead{zhenguan1993@foxmail.com}

\author{Xianxian Cao}

\cortext[cor1]{Corresponding author.}

\address{School of Mathematics and Statistics, Pingdingshan University, Pingdingshan, 467000, China}

\begin{abstract}
In this paper, a second-order linearized discontinuous Galerkin method on general meshes, which treats the backward differentiation formula of order two (BDF2) and Crank-Nicolson schemes as special cases, is proposed for solving the two-dimensional Ginzburg-Landau equations with cubic nonlinearity. By utilizing the discontinuous Galerkin inverse inequality and the mathematical induction method, the unconditionally optimal error estimate in $L^2$-norm is obtained. 
The core of the analysis in this paper resides in the classification and discussion of the relationship between the temporal step size $\tau$ and the spatial step size $h$, specifically distinguishing between the two scenarios of $\tau^2 \leq h^{k+1}$ and $\tau^2 > h^{k+1}$, where $k$ denotes the degree of the discrete spatial scheme. Finally, this paper presents two numerical examples involving various grids and polynomial degrees to verify the correctness of the theoretical results.
 
\end{abstract}


\end{frontmatter}

\thispagestyle{empty}

\numberwithin{equation}{section}
\section{Introduction}\label{section01}

In this paper, we will consider the following initial and boundary problems of two-dimensional Ginzburg-Landau equations:   
\begin{align}
\begin{cases}
	u_t - (\nu + \mathrm{i}\alpha)\Delta u + (\kappa + \mathrm{i}\beta)\vert u \vert^2 u - \gamma u = 0, \quad (\boldsymbol{x},t) \in \Omega \times (0,T], \\
	u(\boldsymbol{x},t) = 0,  \quad (\boldsymbol{x},t) \in \partial\Omega\times(0,T], \\
	u(\boldsymbol{x},0) = u^0(\boldsymbol{x}), \quad \boldsymbol{x} \in \bar{\Omega},\label{202507262339}
\end{cases}
\end{align}
where $\Omega$ denotes a bounded convex polygonal domain in $ \mathbb{R}^2$, $\partial\Omega$ is the boundary of $\Omega$,
$\Delta$ is the Laplace operator, $\mathrm{i}$ represents the the imaginary unit, $\nu>0,\kappa>0,\alpha,\beta,\gamma$ are three given real constants, and $u^0(\boldsymbol{x})$ is a sufficiently smooth complex-valued function with a zero boundary trace. The aforementioned model was first proposed in 1950 by two Soviet physicists, Vitaly Ginzburg and Lev Landau, and has been widely applied in fields such as non-equilibrium hydrodynamic systems \cite{Abarzhi} and physical phase transitions \cite{Hoffmann}. Due to the practical importance of the equations, a large number of scholars have worked on solving them from both analytical and numerical perspectives. In terms of mathematical analysis, for example, Porubov and Velarde \cite{Porubov} obtained three new exact periodic solutions of the complex Ginzburg–Landau equation in terms of the Weierstrass elliptic function. Doering et al. \cite{Doering} investigated existence and regularity of solutions to the generalized complex Ginzburg-Landau equation subject to periodic boundary conditions in various spatial dimensions. Akhmediev et al. \cite{Akhmediev} presented novel stable solutions which are soliton pairs and trains of the 1D complex Ginzburg-Landau equation. For a more detailed mathematical analysis of the Ginzburg-Landau equations, interested readers are advised to refer to recent monographs \cite{Guo}. Due to the nonlinear nature of the equations, it is difficult to obtain their analytical solutions. To date, a large number of effective numerical methods have been developed, such as finite difference method, finite element method, meshless method, virtual element method, and $H^1$-Galerkin method. Xu and Chang \cite{Xu2011} presented three difference schemes of the Ginzburg-Landau Equation in two dimensions. They proved the stability of the two difference schemes by virtue of induction method and linearized analysis. Hu et al.  \cite{Chen2015} 
established fourth-order compact finite difference schemes for the 1D nonlinear Kuramoto-Tsuzuki equation with Neumann boundary conditions and conducted numerical analysis. They then extended the methods to 2D. Hao et al. \cite{Hao2015} proposed a high-order finite difference method for the two-dimensional complex Ginzburg–Landau equation. They proved that the proposed difference scheme was uniquely solvable and unconditionally convergent by energy methods. Shi and Liu \cite{Shi2020} presented a two-grid method (TGM)  for the complex Ginzburg-Landau equation, i.e., the original nonlinear system is analyzed on the coarse grid and then a simple linearized problem on the fine grid is solved. Furthermore, they also deduced the superclose estimation in the $H^1$-norm for the TGM scheme. Yang and Jia \cite{Yang2025} proposed the backward Euler-Galerkin finite element method for the two-dimensional Kuramoto–Tsuzuki equation. They obtained the optimal error estimate in $L^2$-norm without any spatio-temporal restrictions. Li et al. \cite{Li2025} proposed a fast element-free Galerkin (EFG) method for solving the nonlinear complex Ginzburg-Landau equation. With the help of the error-spliting argument, they proved the optimal error estimate in $L^2$ and $H^1$ norms. Wang and Li \cite{Wang2023} considered a linearized time-variable-step second order backward differentiation formula (BDF2) virtual element method for the nonlinear Ginzburg-Landau equation. By using the techniques of the discrete complementary convolution (DOC) kernels and the discrete complementary convolution (DCC) kernels, they derived the optimal error estimate in $L^2$-norm.
They also extended the scheme to coupled Ginzburg-Landau equations (CGLEs) \cite{Li2023}. Shi and Wang \cite{Shi2019} discussed an $H^1$-Galerkin mixed finite element method (MFEM) for the two-dimensional Ginzburg-Landau equation with the bilinear element and zero order Raviart-Thomas element. 

In this paper, based on the second-order $\theta$ scheme in time proposed by Liu et al. \cite{Liu2019} and the polygonal discontinuous Galerkin methods, we proposed a weighted implicit-explicit discontinuous Galerkin methods for two-dimensional Ginzburg-Landau equations. To the best of our  knowledge, this linearized polygonal discontinuous Galerkin scheme has not been presented in the literature. The discontinuous Galerkin algorithm employed in this paper is referred to as the symmetric interior penalty Galerkin (SIPG) method in the literature. This algorithm was first introduced and analyzed by Wheeler \cite{Wheeler1978}, and then was generalized to nonlinear elliptic and parabolic equations by Arnold \cite{Arnold1982}. Discontinuous Galerkin methods were first applied to polygonal meshes obtained by element agglomeration in \cite{Bassi2012}. As far as I know, in order to achieve the unconditional convergence of linearized numerical schemes, i.e., there is no restrictive condition between the time step and the space meshsize, the method commonly used in the literature is the space-time error splitting technique proposed by Li and Sun \cite{Li2013,Li2013b}. Since this analytical method requires the introduction of an additional time-discretized system, it brings a certain degree of complexity to numerical analysis to some extent. In this paper, drawing on the ideas proposed by Sun and Wang  \cite{Sun2017}, we present a relatively simple analytical method. Specifically, this method directly performs a theoretical analysis of the fully discrete scheme without introducing a time-discretized system, thereby simplifying the complex error derivation process used in previous studies. The core idea of the argument is to conduct a classified discussion on the relationship between $\tau$ and $h$. In addition, this paper presents, for the first time, a generalized discontinuous Galerkin inverse inequality. By combining this inequality with a transfer formula, we ultimately establish the optimal-order estimate of the $L^2$-norm for the numerical solution-thus filling the gap in the convergence analysis of discontinuous Galerkin schemes for this class of nonlinear problems.

The structure and content of this paper are organized as follows. In Section \ref{section02}, the derivation process of the fully discrete numerical scheme is presented, and several important preparatory lemmas are proven. In Section \ref{section3}, the theoretical analysis of the weighted implicit-explicit discontinuous finite element method is presented, including its stability and mesh-ratio-free convergence. In Section \ref{section4}, we present two numerical examples to verify the correctness of the theoretical analysis, covering both convex and non-convex mesh partitions. Finally, we summarize the content of this paper in Section \ref{section5}, and briefly discuss potential directions for future research.

\section{Weighted implicit-explicit discontinuous Galerkin
formulation}\label{section02}
\subsection{The variational formulation}
Unless otherwise specified, all functions and vector spaces considered in this paper are complex. For a given open subset $D$ of the domain $\Omega$, let $|\cdot|_{m,p,D}$ and $\|\cdot\|_{m,p,D}$ be the seminorm and norm of the Sobolev space $W^{m,p}(D)$, respectively, where $m \geq 0$ is an integer and $1 \leq p \leq \infty$  a real number. When $p=2$, we denote $W^{m,2}(D)$ by $H^{m}(D)$ and the corresponding seminorm and norm are abbreviated as $|\cdot|_{m,D}$ and $\|\cdot\|_{m,D}$, respectively. Let $H^{m}_{0}(D)$ be the closure of  $C_{0}^{\infty}(\mathrm{\Omega})$ with respect to the norm $\|\cdot\|_{m,D}$. Denote by $(\cdot,\cdot)_{D}$ and $\|\cdot\|_{D}$ the inner product and the corresponding
norm of the Hilbert space $L^{2}(D)$, respectively. When $D=\Omega$, we omit the subscript $\Omega$ in the above norm, seminorm and inner product. Finally, for a strongly measurable function $v:(0,T)\rightarrow X$, we introduce the Bochner space defined as 
$$
L^{p}({(0,T);X})=\big\{v:||v||_{L^{p}((0,T);X)}<\infty\big\},
$$
where $X$ is a complex Banach space. 

By virtue of the notations introduced above, it is easy to derive the weak formulation of the equation \eqref{202507262339}:
for all $0<t\leq T$, find $u\in L^{2}((0,T);H^{1}_{0}(\Omega))\cap L^{2}((0,T);L^{2}(\Omega))$ with the initial condition $u(0)=u^0(\boldsymbol{x})$ such that 
\begin{align}
(u_t,v)+(\nu + \mathrm{i}\alpha)a(u, v)+(\kappa + \mathrm{i}\beta)(\vert u \vert^2 u,v)-\gamma (u,v)=0,\quad \forall ~v\in H^{1}_{0}(\Omega),\label{202507262352}
\end{align}
where $a(u,v):=(\nabla u,\nabla v)$.
\subsection{Discrete settings of discontinuous Galerkin}
Let $\mathcal{T}_{h}$ be a sequence of partitions of $\mathrm{\Omega}$ consisting of arbitrary polygons $K$ with its measure denoted as $\vert K\vert$, $h$ is the spatial mesh size parameter which is defined as $h=\max_{K\in\mathcal{T}_{h}}h_{K}$, where $h_{K}$ is the diameter of $K$. A edge $E$ is defined as a closed subset in $\bar{\Omega}$ such that either there exist distinct mesh elements $K_1,K_2\in\mathcal{T}_{h}$ such that $E = \partial K_1\cap \partial K_2$, or there exists a mesh element $K$ such that $F = \partial K\cap \partial \Omega$. We observe that in the case of general meshes containing non-convex polygon, interior edges are not always parts of hyperplanes. Denote the set of all edges in the partition $\mathcal{T}_{h}$ by $\mathcal{E}_{h}$. Moreover, interior edges are collected in the set $\mathcal{E}^{i}_{h}$ and boundary edges in $\mathcal{E}^{b}_{h}$, so that $\mathcal{E}_{h}=\mathcal{E}^{i}_{h}\cup\mathcal{E}^{b}_{h}$. 
It should be emphasized here that mesh partitions with hanging nodes are allowed.

We also assume that the family of meshes $\mathcal{T}_{h}$ satisfies the regularity conditions outlined in \cite{DiPietroErn2012}, i.e., there exists a real number $\rho\in(0,1)$, independent of $h$ and called the mesh regularity parameter, such that there exists a matching simplicial submesh $\mathfrak{T}_{h}$ that satisfies the following conditions: \\
(i) Shape regularity. For any simplex $\tau\in\mathfrak{T}_{h}$, denoting by $h_\tau$ its diameter
and by $r_\tau$ its inradius, it holds 
$$\rho h_\tau\leq r_\tau;$$
(ii) Contact regularity. For any mesh element $K\in\mathcal{T}_{h}$ and any simplex $\tau\in\mathfrak{T}_{K}$
, where $\mathfrak{T}_{K}$ is the set of simplices contained in
$K$, it holds 
$$\rho h_T\leq h_\tau.$$

Let $v$ be a scalar-valued function defined on $\Omega$ and assume that $v$ is smooth enough to admit on all $E\in \mathcal{E}^{i}_{h}$ a possibly two-valued trace. Then, for all $E\in \mathcal{E}^{i}_{h}$ shared by two adjacent elements $K_1$ and $K_2$, the interior edge $E$ is oriented by means of the unit normal vector $\boldsymbol{n}_E$ pointing from $K_1$ to $K_2$,
i.e., 
$$\boldsymbol{n}_E=\boldsymbol{n}_{K_{1}E}=-\boldsymbol{n}_{K_{2}E},$$
where $\boldsymbol{n}_{K_{i}E}$ denotes the unit normal vector to pointing out of $K_{i}$, $i=1,2$. The average and jump are defined as 
\begin{align}
\{v\} = \frac{1}{2}(v\vert_{K_1}+v\vert_{K_2}),\quad [v] = v\vert_{K_1}-v\vert_{K_2},\notag
\end{align}
respectively. By convention, we can extended the definition of jump and average to edges that belong to the boundary edge $E = \partial K\cap\partial\Omega$: 
$$
\{v\} = [v] = v\vert_{K}.
$$
At this moment, $\boldsymbol{n}_E$ is taken to be the unit outward vector normal to $\partial\Omega$.

Let $k$ be a positive integer, the discontinuous finite element space $V_{h}^{k}$ is choose to be
\begin{align}
	V_{h}^{k}=\{v_{h}\in L^{2}(\Omega): v_{h}|_{K}\in \mathbb{P}_{k}(K),~\forall~K\in\mathcal{T}_{h}\},\notag 
\end{align}
where $\mathbb{P}_{k}(K)$ denotes the space of polynomials of total degree less than or equal to $k$. This finite-dimensional space equipped the 
with the following norm 
\begin{align}
\|v_{h}\|_{\text{DG}}:= \bigg(\sum_{K \in \mathcal{T}_h} \int_K |\nabla v_h|^2\mathrm{d}\boldsymbol{x}+ \sum_{E \in \mathcal{E}_h} \frac{1}{h_E} \int_E  [v_h] ^2\mathrm{d}s\bigg)^{1/2},\label{2507281430}
\end{align}
where $h_E$ is the diameter of the edge $E$ and $\vert \cdot \vert$ denotes the Euclidean norm in $\mathbb{R}^2$, satisfies the discrete Sobolev embedding inequality \cite{PietroErn2010}
\begin{align}
\| v_h\|_{0,p} \leq C  \| v_h \|_{\text{DG}}, \quad \forall ~v_h \in V_{h}^{k}, \notag
\end{align}
where $1 \leq p < \infty$. We also present the discrete Ladyzhenskaya's inequality \cite{GazcaOrozcoKaltenbach2025} that play a crucial role in deriving the optimal error estimate:
\begin{align}
	\| v_h\|_{0,4} \leq C  \| v_h \|_{\text{DG}}^{1/2}\| v_h\|^{1/2}, \quad \forall ~v_h \in V_{h}^{k}, \label{202509042305}
\end{align}

The local $L^2$-orthogonal projector $\Pi^{0,k}_{K}:L^{2}(K)\rightarrow \mathbb{P}_{k}(K)$ is defined as follows: for all $v\in L^2(K)$, the polynomial $\Pi^{0,k}_{K}v$ satisfies
$$(\Pi^{0,k}_{K}v-v,q)_{K}=0,\quad \forall~q\in \mathbb{P}_{k}(K).$$
The global $L^2$-orthogonal projector $\Pi^{0,k}_{h}:L^{2}(\Omega)\rightarrow V_{h}^{k}$ can be easily obtained, i.e., for all $v\in L^{2}(\Omega)$ and all $K\in \mathcal{T}_{h} $
$$(\Pi^{0,k}_{h}v)|_{K}=\Pi^{0,k}_{K}(v|_{K}).$$
According to the exquisite argument in the book of Di Pietro and Droniou \cite{PietroDroniou2020}, The global $L^2$-orthogonal projector $\Pi^{0,k}_{h}$ satisfies the following approximation and boundedness properties:
\begin{align}
&||v-\Pi^{0,k}_{h}v||\leq Ch^{k+1}||v||_{k+1},\quad \forall~v\in H^{k+1}(\Omega),\label{2507281414}\\
&||\Pi^{0,k}_{h}v||_{0,\infty}\leq C||v||_{0,\infty}, \quad \forall~v\in L^{\infty}(\Omega),
\end{align}
respectively.

In addition to the above assumption about the mesh $\mathcal{T}_{h}$, we suppose the mesh is quasi-uniform, i.e.,
\begin{align}
\rho h\leq h_{K},\quad \forall~ K\in \mathcal{T}_{h}.\label{2507281534}
\end{align}
Then, the following global inverse inequality is true \cite{Chave2016}:
\begin{align}
||v_h||_{0,\infty}\leq Ch^{-1}||v_h||, \quad \forall~v_h\in V_h^{k}.\label{2507281409}
\end{align}
\subsection{The fully discrete numerical scheme}
Let $N$ be a positive integer and let $\tau=T/N$ denotes the time step size. For any smooth function $v(\boldsymbol{x},t)$, we always use the following notations: 
\begin{align}
t_{n-\theta}&:=(n-\theta)\tau,\quad v(t_{n-\theta})(\boldsymbol{x}):=v(\boldsymbol{x},t_{n-\theta}), \quad n = 1, 2, \cdots, N-1, N,\quad \theta\in\left[0, 1/2 \right];\\
D_\tau v^{n-\theta} &:= \frac{\left( 3-2\theta \right) v^{n}-(4-4\theta) v^{n-1}+\left(1-2\theta\right) v^{n-2}}{2\tau}, \quad n = 2, 3,\cdots, N-1, N;\\
v^{n-\theta} &:= (1-\theta) v^{n} + \theta v^{n-1}, \quad n = 1, 2,\cdots, N-1, N; \\
\hat{v}^{n-\theta} &:= (2 -\theta) v^{n-1} - (1-\theta) v^{n-2}, \quad n = 2, 3,  \cdots, N-1, N.
\end{align}

In order to approximate the bilinear form $a(u,v)=(\nabla u,\nabla v)$, we employing the following discrete bilinear form:
\begin{align}
a_h(v_h, w_h) &:= \sum_{K \in \mathcal{T}_h} \int_K\nabla v_h \cdot \nabla w_h \, \mathrm{d}x 
- \sum_{E \in \mathcal{E}_h} \int_E \{ \nabla_h v_h \} \cdot \boldsymbol{n}_E [w_h] \, \mathrm{d}s \\
&\quad - \sum_{E \in \mathcal{E}_h} \int_E [v_h] \{ \nabla_h w_h \} \cdot \boldsymbol{n}_E \, \mathrm{d}s 
+ \sum_{E \in \mathcal{E}_h} \frac{\lambda}{h_E}\int_E [v_h][w_h] \, \mathrm{d}s, \quad \forall~v_h, w_h\in V_{h}^{k}, 
\end{align}
where the parameter $\lambda$ is called penalty term, which is a sufficiently large non-negative real number. It is well-known that the discrete bilinear form defined above satisfies the following coercivity and continuity properties:
\begin{align}
C \| v_h \|_{\text{DG}}^{2} \leq a_{h}(v_{h},v_{h})=:\|v_{h}\|^2_{a_h} \leq C \| v_h \|_{\text{DG}}^{2}, \quad \forall~ v_h\in V_h^{k}.\label{2508251529}
\end{align}
In order to obtain the unconditionally optimal order error estimate for the fully discrete numerical scheme, we define  the elliptic projection operator $R_{h}: H^{2}(\Omega)\rightarrow V_{h}^{k}$ such that 
\begin{align}
	a_{h}(R_{h}u,v_{h})=a_h(u,v_{h}),\quad \forall v_{h}\in V_{h}^{k}.\label{2507271753}
\end{align}
According to the classical theory of discontinuous finite elements for solving elliptic problems \cite{Riviere2008}, we can obtain the following projection error:
\begin{align}
	\|R_{h}u-u\|\leq Ch^{k+1}\|u\|_{k+1}, \quad \forall~ u\in  H^{k+1}(\Omega)\label{2507281415}.
\end{align}

With above notations, the weighted implicit-explicit (IMEX) discontinuous Galerkin algorithm is given as follows: Find $u_{h}^{n}\in V_{h}^{k}$ such that for $n=2, 3, \cdots, N-1, N$ 
\begin{align}
	&(D_\tau u_{h}^{n-\theta},v_{h})+(\nu + \mathrm{i}\alpha)a_h(u_h^{n-\theta}, v_h)+(\kappa + \mathrm{i}\beta)(\vert \hat{u}_h^{n-\theta} \vert^2 u_h^{n-\theta},v_h)\notag\\
	&\quad-\gamma (u_h^{n-\theta},v_h)=0,\quad \forall ~v_h\in V_h^k,\label{07271656}
\end{align}
with the initial approximation $u_{h}^0=R_hu^0$. 

Since the above scheme is a three-level method, we need to additionally provide a second-order calculation method for $u(t_1)$. Here, we analyze a backward Euler Galerkin method for this purpose, i.e., $u_h^1$ is the solution of the following equation:
\begin{align}
	\left(\frac{u_h^{1}-u_h^0}{\tau},v_{h}\right)+(\nu + \mathrm{i}\alpha)a_h(u_h^{1}, v_h)+(\kappa + \mathrm{i}\beta)(\vert u_h^{0} \vert^2 u_h^{1},v_h)
     -\gamma (u_h^{1},v_h)=0,\quad \forall ~v_h\in V_h^k.\label{07271841}
\end{align}
\subsection{Some vital results}

\begin{lemma}\label{lemma1}
\text{(\cite{SunSunGao2016,GaoSunSun2015})} Assume that $\mathcal{V}$ is an complex inner product space equipped  with the inner product $(\cdot, \cdot)_{\mathcal{V}}$ and the induced norm $\|\cdot\|_{\mathcal{V}}$. Then, for any $v^0, v^1, \ldots, v^N \in \mathcal{V}$, it holds that 
\begin{align}
\text{Re}(D_{\tau}v^{n-\theta}, v^{n-\theta})_{\mathcal{V}} \geq \frac{1}{4\tau} \left( E^{n} - E^{n-1} \right), \quad 2 \leq n \leq N,
\end{align}
where
\begin{align}
E^{n} = (3-2\theta)\|v^{n}\|_{\mathcal{V}}^2 - (1-2\theta)\|v^{n-1}\|_{\mathcal{V}}^2 + (2-\theta)(1-2\theta)\|v^{n}-v^{n-1}\|_{\mathcal{V}}^2, \quad 1 \leq n \leq N.
\end{align}
In addition, the following inequality holds 
\begin{align}
E^n \geq \frac{1}{1-\theta}\|v^n\|_{\mathcal{V}}^2, \quad 1 \leq n \leq N.
\end{align}	
\end{lemma}
\begin{lemma}\label{lemma2} Assuming that polygon mesh subdivision of $\Omega$ is regular and quasi-uniform and $v\in  H^{k+1}(\Omega)$, there exists a positive constant $C$ independent of the mesh subdivision parameters $h$ and $\tau$, such that 
\begin{align}
\|v-R_hv\|_{0,\infty}+\|R_hv\|_{0,\infty}\leq C;
\end{align}
\begin{proof}
Using the error estimates  \eqref{2507281415} and  \eqref{2507281414} satisfied by the elliptic projection and $L^2$-orthogonal projector, respectively, and the global inverse inequality \eqref{2507281409}, we obtain 
\begin{align*}
\|v-R_hv\|_{0,\infty}&=\|v-\Pi^{0,k}_{h}v+\Pi^{0,k}_{h}v-R_hv\|_{0,\infty}\\
&\leq \|v-\Pi^{0,k}_{h}v\|_{0,\infty}+\|\Pi^{0,k}_{h}v-R_hv\|_{0,\infty}\\
&\leq C\|v\|_{0,\infty}+Ch^{-1}\|\Pi^{0,k}_{h}v-R_hv\|\\
&\leq C\|v\|_{k+1}+Ch^{-1}(\|\Pi^{0,k}_{h}v-v\|+\|v-R_hv\|)\\
&\leq C\|v\|_{k+1}+Ch^{k}\|v\|_{k+1}\\
&\leq C. 
\end{align*}
Furthermore, by virtue of the triangle inequality, we have 
\begin{align*}
\|R_hv\|_{0,\infty}\leq \|v-R_hv\|_{0,\infty}+\|v\|_{0,\infty}\leq C.
\end{align*}
The proof of Lemma \ref{lemma2} is complete.
\end{proof}
\end{lemma}
\begin{lemma}\label{lemma3}
For the energy norm defined in \eqref{2507281430}, the  following generalized discontinuous Galerkin inverse inequality holds: 
\begin{align}
\|v_h\|_{\text{DG}}\leq Ch^{-1}\|v_h\|,\quad \forall~v_h\in V_{h}^{k}.	\label{08292320}
\end{align}
\begin{proof}
\begin{align*}
\|v_h\|_{\text{DG}}&=\bigg(\sum_{K \in \mathcal{T}_h} \int_K |\nabla v_h|^2\mathrm{d}\boldsymbol{x}+ \sum_{E \in \mathcal{E}_h} \frac{1}{h_E} \int_E  [v_h] ^2\mathrm{d}s\bigg)^{1/2}\\
&\leq C\bigg(\sum_{K \in \mathcal{T}_h} h_K^{-2}\int_K | v_h|^2\mathrm{d}\boldsymbol{x}+ \sum_{K \in \mathcal{T}_h} h_K^{-1} \int_{\partial K} |v_h| ^2\mathrm{d}s\bigg)^{1/2}\\
&\leq C\bigg(\sum_{K \in \mathcal{T}_h} h_K^{-2}\int_K | v_h|^2\mathrm{d}\boldsymbol{x}+ \sum_{K \in \mathcal{T}_h} h_K^{-2} \int_{K} |v_h| ^2\mathrm{d}s\bigg)^{1/2}\\
&\leq Ch^{-1}\|v_h\|,
\end{align*}
where we have used the discrete inverse inequality and the  trace inequality on regular mesh sequences  \cite{PietroDroniou2020} 
\begin{align*}
\|\nabla v\|_{K} \leq C h_K^{-1} \|v\|_{K},\quad \| v \|_{\partial K} \leq Ch_K^{-\frac{1}{2}} \| v \|_{K}, \quad \forall ~ v\in \mathbb{P}_{k}(K),
\end{align*} 
the geometric inequality \cite{PietroDroniou2020}
$$Ch_K\leq h_E\leq h_K,\quad E\subset\partial K,$$
together with the quasi-uniform assumption \eqref{2507281534}. 
\end{proof}
\end{lemma}
\begin{lemma}\label{lemma4}
\text{(\cite{HeywoodRannacher1990})} Let $a\geq 0, b \geq 0$, $\{\eta^i\}_{i=1}^{N}$ and $\{\xi^i\}_{i=1}^{N}$ be two series of non-negative real numbers such that
$$
\quad \eta^n + \tau\sum_{i=1}^n \xi^i\leq a + b\tau \sum_{i=1}^n \eta^i, \quad 1 \le n \le N.
$$
Then, when $\tau \le \frac{1}{2b}$, it holds that
$$\eta_n +\tau\sum_{i=1}^n \xi^i\leq a \exp(2b n \tau), \quad 1 \le n \le N.$$
\end{lemma}
\begin{lemma}\label{lemma5}
\text{(\cite{Guan2022})} Suppose that $\mathcal{V}$ is a normed linear space with the norm  $\|\cdot\|_{\mathcal{V}}$, and $v^0, v^1, \ldots, v^N \in \mathcal{V}$. Then, we have 
\begin{align}
	\|v^n\|_{\mathcal{V}} \leq (1+2\theta) \sum_{m=1}^{n} \left\|v^{m-\theta}\right\|_{\mathcal{V}} + 2\theta \|v^0\|_{\mathcal{V}}, \quad 1 \leq n \leq N.\label{2507302329}	
\end{align}
\end{lemma}
\section{Theoretical analysis of the weighted IMEX discontinuous Galerkin method} \label{section3}
\subsection{The stability of the fully discrete numerical solution}  
\begin{theorem}\label{theorem1}
Suppose that $u_{h}^{n}$ is the solution of the numerical scheme \eqref{07271656}-\eqref{07271841}, when the time stepsize $\tau$ satisfies $\max\{0,\gamma\}\tau\leq \frac{1}{16}$, we have    
\begin{align}
\|u_{h}^{n}\|\leq C_1 \|u_{h}^{0}\|,\quad 1\leq n \leq N,
\end{align}
where
$$
C_1=\left(\exp\left(32\max\{\gamma,0\}T\right) \left(24+\frac{128}{7}\max\{\gamma,0\}\right)\right)^{1/2}.
$$
\begin{proof}
(I) Setting $v_h=u_h^{1}$ in \eqref{07271841}, it holds that 
\begin{align}
	\bigg(\frac{u_h^{1}-u_h^0}{\tau},u_h^{1}\bigg)+(\nu + \mathrm{i}\alpha)\|u_h^{1}\|_{a_h}^2+(\kappa + \mathrm{i}\beta)(\vert u_h^{0} \vert^2 u_h^{1},u_h^{1})
	-\gamma \|u_h^{1}\|^2=0.\label{07311033}
\end{align}
Noticing the fact that 
\begin{align}
\text{Re}\left(\frac{u_h^{1}-u_h^0}{\tau},u_h^{1}\right) &= \frac{1}{2\tau}\left(\|u_{h}^{1}\|^2-\|u_{h}^{0}\|^2+\|u_{h}^{1}-u_{h}^{0}\|^2\right)\notag\\
&\geq \frac{1}{2\tau}\left(\|u_{h}^{1}\|^2-\|u_{h}^{0}\|^2\right), \notag 
\end{align}
and taking the real parts on both the right and left hand sides of \eqref{07311033}, we have
\begin{align}
(1-2\gamma\tau)\|u_{h}^{1}\|^2\leq \|u_{h}^{0}\|^2. \notag 
\end{align}
Obviously, when $\gamma\leq0$, it follows that 
\begin{align}
\|u_{h}^{1}\|\leq \|u_{h}^{0}\|\leq C_1 \|u_{h}^{0}\|.\label{2507311515}
\end{align}
When $\gamma>0$ and $\tau\leq \frac{1}{16\gamma}$, it holds that 
\begin{align}
\|u_{h}^{1}\|\leq \sqrt{\frac{1}{1-2\gamma\tau}}\|u_{h}^{0}\|\leq \sqrt{\frac{8}{7}}\|u_{h}^{0}\|\leq C_1 \|u_{h}^{0}\|^2.\label{2507311518}
\end{align}
(II) Taking $v_h=u_h^{n-\theta}$ in \eqref{07271656} and considering the real parts of both sides of the equation yield that 
\begin{align}
	&\text{Re}\left(D_\tau u_{h}^{n-\theta},u_h^{n-\theta}\right)+\nu \|u_h^{n-\theta}\|^2_{a_h}+\kappa\left(\vert\hat{u}_h^{n-\theta}\vert^2u_h^{n-\theta},u_h^{n-\theta}\right)-\gamma \|u_h^{n-\theta}\|^2=0.\label{07311555}
\end{align}
Utilizing Lemma \ref{lemma1}, it follows that 
\begin{align}
\frac{1}{4\tau}\left(F^n-F^{n-1}\right) \leq \gamma\|u_h^{n-\theta}\|^2,\quad 2\leq n \leq N, \label{2507311309}
\end{align}
where 
\begin{align}
F^{n} &= (3-2\theta)\|u_h^{n}\|^2 - (1-2\theta)\|u_h^{n-1}\|^2 + (2-\theta)(1-2\theta)\|u_h^{n}-u_h^{n-1}\|^2\notag\\
&\geq \frac{1}{1-\theta} \|u_h^{n}\|^2, \quad 2 \leq n \leq N.\label{2507311341}
\end{align}
Replacing $n$ by $m$ and summing up for $m$ from $2$ to $n$ on both sides of \eqref{2507311309} and employing \eqref{2507311341}, we arrive at
\begin{align}
\|u_h^{n}\|^2\leq (1-\theta)F^n\leq F^n &\leq F^1+4\gamma
\tau\sum\limits_{m=2}^{n}\|u_h^{m-\theta}\|^2,\quad 2 \leq n \leq N.
\end{align}
Noticing 
\begin{align}
F^{1} &= (3-2\theta)\|u_h^{1}\|^2 - (1-2\theta)\|u_h^{0}\|^2 + (2-\theta)(1-2\theta)\|u_h^{1}-u_h^{0}\|^2\notag\\
&\leq 3\|u_h^{1}\|^2+2\|u_h^{1}-u_h^{0}\|^2\notag\\
&\leq 7\|u_h^{1}\|^2+4\|u_h^{0}\|^2\notag\\
&\leq 12\|u_h^{0}\|^2.
\end{align}
We have 
\begin{align}
	\|u_h^{n}\|^2\leq 12\|u_h^{0}\|^2+4\gamma
	\tau\sum\limits_{m=2}^{n}\|u_h^{m-\theta}\|^2,\quad 2 \leq n \leq N.\label{202508282115}
\end{align}
When $\gamma\leq 0$, it follows that 
\begin{align}
\|u_h^{n}\|^2\leq 12\|u_h^{0}\|^2.\label{08282210}
\end{align}
When $\gamma > 0$, one can easily obtain from \eqref{202508282115} that 
\begin{align}
	\|u_h^{n}\|^2&\leq 12\|u_h^{0}\|^2+4\gamma
	\tau\sum\limits_{m=2}^{n}\|u_h^{m-\theta}\|^2\label{202508282117}\notag\\
    &\leq 12\|u_h^{0}\|^2+4\gamma\tau\sum\limits_{m=2}^{n}\|(1-\theta) u_h^{m} + \theta u_h^{m-1}\|^2\notag\\
    &\leq 12\|u_h^{0}\|^2+8\gamma\tau\sum\limits_{m=2}^{n}\left(\|(1-\theta) u_h^{m}\|^2 + \|\theta u_h^{m-1}\|^2\right)\notag\\
    &\leq 12\|u_h^{0}\|^2+8\gamma\tau\sum\limits_{m=2}^{n}\left(\|u_h^{m}\|^2 + \| u_h^{m-1}\|^2\right)\notag\\
    &\leq 12\|u_h^{0}\|^2+8\gamma\tau\|u_h^{n}\|^2+16\gamma\tau\sum\limits_{m=2}^{n-1}\|u_h^{m}\|^2+8\gamma\tau\|u_h^{1}\|^2, \quad 2 \leq n \leq N.
\end{align}
That is to say 
\begin{align}
\left(1-8\gamma\tau\right) \|u_h^{n}\|^2 \leq \left(12+\frac{64}{7}\gamma\right)\|u_h^{0}\|^2+16\gamma\tau\sum\limits_{m=2}^{n-1}\|u_h^{m}\|^2, \quad 2 \leq n \leq N.
\end{align}
When $\tau\leq \frac{1}{16\gamma}$, we have 
\begin{align}
\|u_h^{n}\|^2\leq \left(24+\frac{128}{7}\gamma\right)\|u_h^{0}\|^2+32\gamma\tau\sum\limits_{m=2}^{n-1}\|u_h^{m}\|^2, \quad 2 \leq n \leq N.\label{08282211}
\end{align}
With the help of estimates \eqref{08282210} and \eqref{08282211}, we have 
 \begin{align}
 	\|u_h^{n}\|^2\leq \left(24+\frac{128}{7}\max\{\gamma,0\}\right)\|u_h^{0}\|^2+32\max\{\gamma,0\}\tau\sum\limits_{m=2}^{n-1}\|u_h^{m}\|^2, \quad 2 \leq n \leq N.
 \end{align}
By virtue of the discrete Gronwall's inequality given in Lemma \ref{lemma2},   
we have 
\begin{align}
\|u_h^{n}\|^2\leq\exp\left(32\max\{\gamma,0\}T\right) \left(24+\frac{128}{7}\max\{\gamma,0\}\right)\|u_h^{0}\|^2\leq C_1^2\|u_h^{0}\|^2
,\quad 2 \leq n \leq N. \label{2507311522}
\end{align}
All this completes the proof. 
\end{proof}
\end{theorem}
\subsection{The convergence of the fully discrete numerical scheme} 
\begin{theorem}\label{lemma3.1}
Let $u_h^1$ and $u(t_{1})$ be the solutions of the problem \eqref{202507262339} and the fully discrete weighted IMEX discontinuous Galerkin scheme \eqref{07271841}, respectively.  Denote 
\begin{align}	
&\eta_h^1=R_hu^{1}-u_h^1,\quad \eta_h^0=R_hu^{0}-u_h^0=0,\notag\\
&\xi^1=u^1-R_hu^{1},\quad \xi^0=u^0-R_hu^{0}\notag.
\end{align}
Then, when $\max\{0,\gamma\}\tau\leq \frac{1}{4}$, there exists a positive constant $C_2$ such that 
\begin{align}
\|\eta_h^1\|+\tau\|\eta_h^1\|_{\text{DG}}\leq C_2\left(\tau^2
+h^{k+1}\right).\label{2508291009}
\end{align}
\end{theorem}
\begin{proof}
At $t=t_{1}$, we have from \eqref{202507262339} that 
\begin{align}
	&\quad \left(\frac{u^{1}-u^0}{\tau},v_{h}\right)+(\nu + \mathrm{i}\alpha)a_h(u^{1}, v_h)+(\kappa + \mathrm{i}\beta)(\vert u^{0} \vert^2 u^{1},v_h)
	-\gamma (u^{1},v_h)\notag\\
	&=\left(\frac{u^{1}-u^0}{\tau}-u_t(t_1),v_{h}\right)+(\kappa + \mathrm{i}\beta)\left(\vert u^{0} \vert^2 u^{1}-|u^1|^2u^1,v_h\right),\quad \forall ~v_h\in V_h^k.\label{07311659}
\end{align}
Subtracting \eqref{07271841} from \eqref{07311659}, it holds that 
\begin{align}
&\quad \left(\frac{\eta_h^{1}}{\tau},v_{h}\right)+(\nu + \mathrm{i}\alpha)a_h(\eta_h^{1}, v_h)+(\kappa + \mathrm{i}\beta)(\vert u^{0} \vert^2 u^{1}-\vert u_h^{0} \vert^2 u_h^{1},v_h)
-\gamma (\eta_h^{1},v_h)\notag\\
&=\left(\frac{u^{1}-u^0}{\tau}-u_t(t_1),v_{h}\right)+(\kappa + \mathrm{i}\beta)\left(\vert u^{0} \vert^2 u^{1}-|u^1|^2u^1,v_h\right)-\left(\frac{\xi^{1}-\xi^0}{\tau},v_{h}\right)+\gamma(\xi^{1},v_h).\label{07311745}	
\end{align}
Taking $v_h=\eta_h^{1}$, multiplying both sides of the equation by $\tau$ and considering the real parts of both sides of the equation yield that 
\begin{align}
&\|\eta_h^{1}\|^2+\nu\tau\|\eta_h^{1}\|_{a_h}^2-\gamma \tau\|\eta_h^{1}\|^{2}=-\tau\text{Re}\left((\kappa + \mathrm{i}\beta)(\vert u^{0} \vert^2 u^{1}-\vert u_h^{1} \vert^2 u_h^{1},\eta_h^{1})\right)+\tau\text{Re}\left(\frac{u^{1}-u^0}{\tau}-u_t(t_1),\eta_h^{1}\right)\notag\\
&\quad+\tau\text{Re}\left((\kappa + \mathrm{i}\beta)(\vert u^{0} \vert^2 u^{1}-|u^1|^2u^1,\eta_h^{1})\right)-\tau\text{Re}\left(\frac{\xi^{1}-\xi^0}{\tau},\eta_h^{1}\right)+\gamma\tau\text{Re}\left(\xi^{1},\eta_h^{1}\right)=:\sum\limits_{i=1}^5A_{i}\label{2508011807}.
\end{align}
Now, we estimate every term in the right-hand side of Equation \eqref{2508011807}. As for $A_1$, it follows that 
\begin{align}
A_1 &= -\tau\text{Re}\left((\kappa + \mathrm{i}\beta)(\vert u^{0} \vert^2 u^{1}-\vert u_h^{0} \vert^2 u_h^1-\eta_h^1,\eta_h^{1})\right)\notag\\
&= -\tau\text{Re}\left(\left(\kappa + \mathrm{i}\beta\right)(\vert u^{0} \vert^2 \left(\xi^1+R_hu^1\right)-\vert u_h^{0} \vert^2 \left(R_hu^1-\eta_h^1\right),\eta_h^{1})\right)\notag\\
& = -\tau\text{Re}\left(\left(\kappa + \mathrm{i}\beta\right)\left(\left(\vert u^{0} \vert^2-\vert u_h^{0} \vert^2\right) R_hu^1+\vert u_h^{0} \vert^2\xi^1,\eta_h^{1}\right)\right)-\tau\kappa\left(\vert u_h^{0} \vert^2\eta_h^1,\eta_h^1\right)\notag\\
&\leq -\tau\text{Re}\left(\left(\kappa + \mathrm{i}\beta\right)\left(\left(\vert u^{0} \vert^2-\vert u_h^{0} \vert^2\right) R_hu^1+\vert u_h^{0} \vert^2\xi^1,\eta_h^{1}\right)\right)\notag\\
& \leq \frac{1}{6}\|\eta_h^{1}\|^2+C\|\xi^1\|^2+\|\xi^0\|^2\notag\\
& \leq \frac{1}{6}\|\eta_h^{1}\|^2+Ch^{2k+2},\label{202508251219}
\end{align}
where we have used the fact that 
\begin{align}
\vert u^{0} \vert^2-\vert u_h^{0} \vert^2=\text{Re}\left(\left(u^{0}-u_h^{0}\right)\left(u^{0}+u_h^{0}\right)^*\right),\notag
\end{align}
where the symbol $*$ denotes the complex conjugate operation on a complex number.
By the Taylor expansions with the integral remainder, it holds that 
\begin{align}
A_2+A_3 &=\tau\text{Re}\left(\frac{u^{1}-u^0}{\tau}-u_t(t_1),\eta_h^{1}\right)+\tau\text{Re}\left((\kappa + \mathrm{i}\beta)(\vert u^{0} \vert^2 u^{1}-|u^1|^2u^1,\eta_h^{1})\right)\notag\\
&\leq C\tau\left\|\frac{u^{1}-u^0}{\tau}-u_t(t_1)\right\|\cdot\|\eta_h^{1}\|+C\tau \left\| \vert u^{0} \vert^2 u^{1}-|u^1|^2u^1 \right\|\cdot\|\eta_h^{1}\|\notag\\
&\leq C\tau^2 \left\|\frac{u^{1}-u^0}{\tau}-u_t(t_1)\right\|^2+C\tau^2\left\| \vert u^{0} \vert^2 u^{1}-|u^1|^2u^1 \right\|^2+\frac{1}{6}\|\eta_h^{1}\|^2\notag\\
&\leq C\tau^4+\frac{1}{6}\|\eta_h^{1}\|^2.\label{202508251220}
\end{align}
In the last, by virtue of the Cauchy-Schwartz inequality, it arrives at 
\begin{align}
A_4+A_5&=-\tau\text{Re}\left(\frac{\xi^{1}-\xi^0}{\tau},\eta_h^{1}\right)+\gamma\tau\text{Re}\left(\xi^{1},\eta_h^{1}\right)\notag\\
&\leq C\left\|\frac{\xi^{1}-\xi^0}{\tau}\right\|^2+C\|\xi^1\|^2+\frac{1}{6}\|\eta_h^{1}\|^2\notag\\
&\leq Ch^{2k+2}+\frac{1}{6}\|\eta_h^{1}\|^2. \label{202508251221}
\end{align} 
Combining the bounds above and substituting \eqref{202508251219}-\eqref{202508251221} into \eqref{2508011807}, we have for sufficiently small $\tau$
\begin{align}
\frac{1}{2}\|\eta_h^{1}\|^2+\nu\tau\|\eta_h^{1}\|_{a_h}^2\leq C\tau^4+Ch^{2k+2}+\gamma\tau\|\eta_h^{1}\|^2.
\end{align}
When $\gamma\leq0$, noting the norm equivalence in \eqref{2508251529}, we have 
\begin{align}
\|\eta_h^{1}\|^2+\tau\|\eta_h^{1}\|_{\text{DG}}^2\leq c_{1}\left(\tau^2+h^{k+1}\right)^2.
\end{align}
When $\gamma>0$ and $\tau\gamma\leq\frac{1}{4}$, it holds that 
\begin{align}
	\frac{1}{4}\|\eta_h^{1}\|^2+\tau\|\eta_h^{1}\|_{\text{DG}}^2\leq\left(\frac{1}{2}-\gamma\tau\right)\|\eta_h^{1}\|^2+\tau\|\eta_h^{1}\|_{\text{DG}}^2\leq C\left(\tau^2+h^{k+1}\right)^2,
\end{align}
i.e., 
\begin{align}
	\|\eta_h^{1}\|^2+\tau\|\eta_h^{1}\|_{\text{DG}}^2\leq c_{2}\left(\tau^2+h^{k+1}\right)^2.
\end{align}
Choosing $C_2=\max(c_1,c_2)$ imply the truth of inequality \eqref{2508291009}. 
\end{proof}

\begin{theorem}
Suppose that $u_h^n$ and $u(t_{n})$ be the solutions of the continuous problem \eqref{202507262339} and the fully discrete numerical scheme \eqref{07271656}-\eqref{07271841}, respectively, then there exist two positive constants $\tau_1$ and $h_{1}$ such that when $\tau\leq\tau_1$ and $h\leq h_1$, we have 
\begin{align}
\|\eta_h^n\|+\tau\|\eta_h^n\|_{\text{DG}}\leq C_3\left(\tau^2
+h^{k+1}\right), \quad 0 \leq n \leq N,\label{2508292249}
\end{align}
where $$C_3=\max\{C_2,c_3\},\quad \eta_h^n=R_hu^{n}-u_h^n,\quad \xi^n=u^n-R_hu^{n},\quad 2 \leq n \leq N,$$
and $c_3$ is given in \eqref{95056}.
Moreover, with the help of triangle inequality, from \eqref{2508292249}, we can immediately obtain the optimal error estimate in $L^2$-norm.
\begin{align}
\|u^n-u_{h}^{n}\|\leq C(\tau^2+h^{k+1}),\quad 0 \leq n \leq N.
\end{align}
\end{theorem}
\begin{proof}
In the whole process, the mathematical induction method is employing to prove \eqref{2508292249}. According to the conclusion in Lemma \ref{lemma3.1}, we can easily obtain that the \eqref{2508292249} is true for the case $n=0,1$. Next, we assume that \eqref{2508292249} holds for $n$ from $0$ to $m-1~(m\geq2)$, this is to say 
\begin{align}
	\|\eta_h^n\|+\tau\|\eta_h^n\|_{\text{DG}}\leq C_3\left(\tau^2
	+h^{k+1}\right), \quad 0 \leq n \leq m-1.\label{2508292334}
\end{align}
Below, we will prove the following inequality by dividing it into two cases
\begin{align}
\|\eta_h^n\|_{\text{DG}}\leq 1,\quad 0 \leq n \leq m-1.
\end{align}
$\mathbf{Case~I}: \tau^2\leq h^{k+1}$\\
In this case, noting \eqref{08292320}, we have
\begin{align}
\|\eta_h^n\|_{\text{DG}}\leq Ch^{-1}\|\eta_h^n\|\leq Ch^{k}\leq 1,
\quad 0 \leq n \leq m-1.\end{align}
$\mathbf{Case~II}: \tau^2> h^{k+1}$\\
It follows from \eqref{2508292334} that 
\begin{align}
\|\eta_h^n\|_{\text{DG}}\leq C_{3}\tau\leq 1, \quad 0 \leq n \leq m-1.
\end{align}
With the above preparations, we will next prove that the inequality \eqref{2508292334} also holds when $n = m$. For this purpose, consider the equation \eqref{202507262339} at $t=t_{n-\theta}$ and since the scheme is consistent, we can easily obtain the equation satisfied by the exact solution
\begin{align}
	&\quad(D_\tau u^{n-\theta},v_{h})+(\nu + \mathrm{i}\alpha)a_h(u^{n-\theta}, v_h)+(\kappa + \mathrm{i}\beta)(\vert \hat{u}^{n-\theta} \vert^2 u^{n-\theta},v_h)\notag\\
	&=\gamma (u^{n-\theta},v_h)+\gamma (u(t_{n-\theta})-u^{n-\theta},v_h)+(D_\tau u^{n-\theta}-u(t_{n-\theta}),v_{h})+(\nu + \mathrm{i}\alpha)a_h(u^{n-\theta}-u(t_{n-\theta}), v_h),\notag\\
	&\quad+(\kappa + \mathrm{i}\beta)(\vert \hat{u}^{n-\theta} \vert^2 u^{n-\theta}-\vert u(t_{n-\theta}) \vert^2 u(t_{n-\theta}),v_h),\quad 2\leq n\leq m,\quad \forall ~v_h\in V_h^k. \label{0830004}
\end{align}
Subtracting \eqref{07271656} from \eqref{0830004},  we have the system of error equation
\begin{align}
	&\quad(D_\tau \eta_h^{n-\theta},v_{h})+(\nu + \mathrm{i}\alpha)a_h(\eta_h^{n-\theta}, v_h)+(\kappa + \mathrm{i}\beta)(\vert \hat{u}^{n-\theta} \vert^2 u^{n-\theta}-\vert \hat{u}_h^{n-\theta} \vert^2 u_h^{n-\theta},v_h)\notag\\
	&=\gamma (\eta_h^{n-\theta},v_h)+\gamma (u(t_{n-\theta})-u^{n-\theta},v_h)+(D_\tau u^{n-\theta}-u(t_{n-\theta}),v_{h})+(\nu + \mathrm{i}\alpha)a_h(u^{n-\theta}-u(t_{n-\theta}), v_h)\notag\\
	&\quad+(\kappa + \mathrm{i}\beta)(\vert \hat{u}^{n-\theta} \vert^2 u^{n-\theta}-\vert u(t_{n-\theta}) \vert^2 u(t_{n-\theta}),v_h)\notag\\
	&\quad-(D_\tau \xi^{n-\theta},v_{h})+\gamma (\xi^{n-\theta},v_h),\quad 2\leq n\leq m,\quad \forall ~v_h\in V_h^k, \label{0830007}
\end{align}
where we have used the fact that $(\nu + \mathrm{i}\alpha)a_h(\xi^{n-\theta}, v_h)=0$.
Let $v_h=\eta_h^{n-\theta}$ in \eqref{0830007} and take the real part of both sides of the resulting equation, we have 
\begin{align}
&\quad\text{Re}\left(D_\tau \eta_h^{n-\theta},\eta_h^{n-\theta}\right)+\nu\|\eta_h^{n-\theta}\|_{a_h}^2-\gamma \|\eta_h^{n-\theta}\|^{2} =-\text{Re}\left((\kappa + \mathrm{i}\beta)\left(\vert \hat{u}^{n-\theta} \vert^2 u^{n-\theta}-\vert \hat{u}_h^{n-\theta} \vert^2 u_h^{n-\theta},\eta_h^{n-\theta}\right)\right)\notag\\
&=\gamma\text{Re}(u(t_{n-\theta})-u^{n-\theta},\eta_h^{n-\theta})+\text{Re}(D_\tau u^{n-\theta}-u(t_{n-\theta}),\eta_h^{n-\theta})+\text{Re}\left((\nu + \mathrm{i}\alpha)a_h(u^{n-\theta}-u(t_{n-\theta}), \eta_h^{n-\theta})\right)\notag\\
&\quad+\text{Re}\left((\kappa + \mathrm{i}\beta)(\vert \hat{u}^{n-\theta} \vert^2 u^{n-\theta}-\vert u(t_{n-\theta}) \vert^2 u(t_{n-\theta}),\eta_h^{n-\theta})\right)\notag\\
&\quad-\text{Re}\left(D_\tau \xi^{n-\theta},\eta_h^{n-\theta}\right)+\gamma \text{Re}\left(\xi^{n-\theta},\eta_h^{n-\theta}\right),\quad 2\leq n\leq m. \label{0831007}
\end{align}
Next, we estimate each term in the above equality. For this purpose, noting that 
\begin{align}
\vert \hat{u}^{n-\theta} \vert^2 u^{n-\theta}-\vert \hat{u}_h^{n-\theta} \vert^2 u_h^{n-\theta}&=\vert \hat{u}^{n-\theta} \vert^2\left(\xi^{n-\theta}+R_hu^{n-\theta}\right)-\vert \hat{u}_h^{n-\theta}\vert^2\left(R_hu^{n-\theta}-\eta_h^{n-\theta}\right)\notag\\
&= \left(\vert \hat{u}^{n-\theta} \vert^2-\vert \hat{u}_h^{n-\theta} \vert^2\right)R_hu^{n-\theta}+\vert \hat{u}^{n-\theta} \vert^2\xi^{n-\theta}+\vert\hat{u}_h^{n-\theta}\vert^2\eta_h^{n-\theta}\notag\\
&=\text{Re}\left((\hat{u}^{n-\theta}-\hat{u}_h^{n-\theta})(\hat{u}^{n-\theta}+\hat{u}_h^{n-\theta})^*\right)R_hu^{n-\theta}+\vert \hat{u}^{n-\theta} \vert^2\xi^{n-\theta}+\vert\hat{u}_h^{n-\theta}\vert^2\eta_h^{n-\theta}\notag\\
&=\text{Re}\left((\hat{\eta}_h^{n-\theta}+\hat{\xi}^{n-\theta})(\hat{u}^{n-\theta}+R_h\hat{u}^{n-\theta}-\hat{\eta}_h^{n-\theta})^*\right)R_hu^{n-\theta}\notag\\
&\quad+\vert \hat{u}^{n-\theta} \vert^2\xi^{n-\theta}+\vert\hat{u}_h^{n-\theta}\vert^2\eta_h^{n-\theta},\quad 2\leq n\leq m.
\end{align}
Then, it holds that 
\begin{align}
	&\quad-\text{Re}\left((\kappa + \mathrm{i}\beta)\left(\vert \hat{u}^{n-\theta} \vert^2 u^{n-\theta}-\vert \hat{u}_h^{n-\theta} \vert^2 u_h^{n-\theta},\eta_h^{n-\theta}\right)\right)\notag\\
	&\leq C(\|\hat{\eta}_h^{n-\theta}\|_{0,4}^2\|\eta_h^{n-\theta}\|\|R_hu^{n-\theta}\|_{0,\infty}+\|\hat{\eta}_h^{n-\theta}\|\|\eta_h^{n-\theta}\|\|\hat{u}^{n-\theta}+R_h\hat{u}^{n-\theta}\|_{0,\infty}\notag\\
	&\quad+\|\hat{\xi}^{n-\theta}\|\|\eta_h^{n-\theta}\|\|R_hu^{n-\theta}\|_{0,\infty}\|\hat{u}^{n-\theta}+R_h\hat{u}^{n-\theta}\|_{0,\infty}\notag\\
	&\quad+\|\hat{\xi}^{n-\theta}\|_{0,\infty}\|R_hu^{n-\theta}\|_{0,\infty}\|\eta_h^{n-\theta}\|\|\hat{\eta}_h^{n-\theta}\|+\|\hat{u}^{n-\theta}\|_{0,\infty}^2\|\xi^{n-\theta}\|\|\eta_h^{n-\theta}\|)\notag\\
	&\leq C(\|\eta_h^{n}\|^2+\|\eta_h^{n-1}\|^2+\|\eta_h^{n-2}\|^2+h^{2k+2}),\quad 2\leq n\leq m,
\end{align}
where we have used \eqref{202509042305} and the facts that 
\begin{align}
\|\hat{\eta}_h^{n-\theta}\|_{0,4}^2\|\eta_h^{n-\theta}\|&\leq C \|\hat{\eta}_h^{n-\theta}\|_{\text{DG}}\|\hat{\eta}_h^{n-\theta}\|\|\eta_h^{n-\theta}\|\notag\\
&\leq C\|\hat{\eta}_h^{n-\theta}\|\|\eta_h^{n-\theta}\|\notag\\
&\leq C(\|\eta_h^{n}\|^2+\|\eta_h^{n-1}\|^2+\|\eta_h^{n-2}\|^2),\quad 2\leq n\leq m.\notag
\end{align}
By virtue of the Taylor formula, we can immediately obtain the 
\begin{align}
&\quad\gamma\text{Re}(u(t_{n-\theta})-u^{n-\theta},\eta_h^{n-\theta})+\text{Re}(D_\tau u^{n-\theta}-u(t_{n-\theta}),\eta_h^{n-\theta})+\text{Re}\left((\nu + \mathrm{i}\alpha)a_h(u^{n-\theta}-u(t_{n-\theta}), \eta_h^{n-\theta})\right)\notag\\
&\quad\quad+\text{Re}\left((\kappa + \mathrm{i}\beta)(\vert \hat{u}^{n-\theta} \vert^2 u^{n-\theta}-\vert u(t_{n-\theta}) \vert^2 u(t_{n-\theta}),\eta_h^{n-\theta})\right)\notag\\
&\leq C(\tau^4+\|\eta_{h}^{n}\|^2+\|\eta_{h}^{n-1}\|^2),\quad 2\leq n\leq m.
\end{align}
Utilizing the approximation properties of Ritz projection, we have 
\begin{align}
&-\text{Re}\left(D_\tau \xi^{n-\theta},\eta_h^{n-\theta}\right)+\gamma \text{Re}\left(\xi^{n-\theta},\eta_h^{n-\theta}\right)\leq C(h^{2k+2}+\|\eta_{h}^{n}\|^2+\|\eta_{h}^{n-1}\|^2),\quad 2\leq n\leq m.
\end{align}
Substituting the above inequality into Equation \eqref{0831007}, we can easily obtain
\begin{align}
&\quad\text{Re}\left(D_\tau \eta_h^{n-\theta},\eta_h^{n-\theta}\right)+\nu\|\eta_h^{n-\theta}\|_{a_h}^2-\gamma \|\eta_h^{n-\theta}\|^{2}\notag\\
&\leq C(\tau^4+h^{2k+2}+\|\eta_h^{n}\|^2+\|\eta_h^{n-1}\|^2+\|\eta_h^{n-2}\|^2),\quad 2\leq n\leq m.
\end{align}
Furthermore, employing Lemma \ref{lemma1}, we have
\begin{align}
\frac{1}{4\tau} \left(G^{n}-G^{n-1} \right)+\nu\|\eta_h^{n-\theta}\|_{a_h}^2&\leq\text{Re}\left(D_\tau \eta_h^{n-\theta},\eta_h^{n-\theta}\right)+\nu\|\eta_h^{n-\theta}\|_{a_h}^2\notag\\
&\leq C(\tau^4+h^{2k+2}+\|\eta_h^{n}\|^2+\|\eta_h^{n-1}\|^2+\|\eta_h^{n-2}\|^2),\quad 2\leq n\leq m,\label{202509042339}
\end{align}
where 
\begin{align}
G^{n} &= (3-2\theta)\|\eta_h^{n}\|^2 - (1-2\theta)\|\eta_h^{n-1}\|^2 + (2-\theta)(1-2\theta)\|\eta_h^{n}-\eta_h^{n-1}\|^2\notag\\
&\geq \frac{1}{1-\theta}\|\eta_h^n\|^2,\quad 1 \leq n \leq m.\label{9502}
\end{align}
Replacing $n$ with $j$ in \eqref{202509042339} and summing the resulting inequality over $j$ from $2$ to $n$ and multiplying both sides by $4\tau$, we have
\begin{align}
G^{n}+C\tau\sum\limits_{j=2}^{n}\|\eta_h^{j-\theta}\|_{a_h}^2&\leq
G^{1}+C(\tau^4+h^{2k+2})+C\tau\sum\limits_{j=2}^{n}\left(\|\eta_h^{j}\|^2+\|\eta_h^{j-1}\|^2+\|\eta_h^{j-2}\|^2\right)\notag\\
&\leq G^{1}+C(\tau^4+h^{2k+2})+C\tau\sum\limits_{j=2}^{n}\|\eta_h^{j}\|^2+C\|\eta_h^{1}\|^2\notag\\
&\leq G^{1}+C(\tau^4+h^{2k+2})+C\tau\sum\limits_{j=2}^{n}\|\eta_h^{j}\|^2, \quad 2\leq n\leq m,
\end{align}
where we have used Theorem \ref{lemma3.1}. Moreover, applying \eqref{9502} and the equivalence of norms \eqref{2508251529}, we can conclude that 
\begin{align}
	\|\eta_h^{n}\|^2+C\tau\sum\limits_{j=2}^{n}\|\eta_h^{j-\theta}\|_{\text{DG}}^2\leq G^{1}+ C(\tau^4+h^{2k+2})+C\tau\sum\limits_{j=2}^{n}\|\eta_h^{j}\|^2, \quad 2\leq n\leq m.
\end{align}
By using the Gronwall's inequality and the simple result
\begin{align}
G^{1} = (3-2\theta)\|\eta_h^{1}\|^2 - (1-2\theta)\|\eta_h^{0}\|^2 + (2-\theta)(1-2\theta)\|\eta_h^{1}-\eta_h^{0}\|^2\leq C(\tau^4+h^{2k+2}),\notag
\end{align}
it holds that 
\begin{align}
\|\eta_h^{m}\|^2+C\tau\sum\limits_{j=2}^{m}\|\eta_h^{j-\theta}\|_{\text{DG}}^2\leq C(\tau^4+h^{2k+2}).\label{95032}
\end{align}
At last, employing Lemma \ref{lemma5} and \eqref{95032}, it can be immediately obtain that 
\begin{align}
\tau\|\eta_h^{m}\|_{\text{DG}}&\leq (1+2\theta) \tau\sum_{j=1}^{m} \left\|\eta_h^{j-\theta}\right\|_{\text{DG}}\notag\\
&\leq (1+2\theta) \sum_{j=1}^{m}\sqrt{\tau}\cdot \left(\sqrt{\tau}\left\|\eta_h^{j-\theta}\right\|_{\text{DG}}\right)\notag\\
&\leq (1+2\theta) \sqrt{n\tau} \sqrt{\tau\sum\limits_{j=1}^{m}\|\eta_h^{j-\theta}\|_{\text{DG}}^2}\notag\\
&\leq (1+2\theta) \sqrt{T} \sqrt{\tau\sum\limits_{j=1}^{m}\|\eta_h^{j-\theta}\|_{\text{DG}}^2}\notag\\
&\leq C (\tau^2+h^{k+1}).\notag
\end{align}
Therefore, we can obtain 
\begin{align}
	\|\eta_h^m\|+\tau\|\eta_h^m\|_{\text{DG}}\leq c_3\left(\tau^2
	+h^{k+1}\right)\leq C_3\left(\tau^2
	+h^{k+1}\right).\label{95056}
\end{align}
All this complete the proof.
\end{proof}
\section{Numerical examples}\label{section4}
In this section, we provide two numerical example to verify the theoretical analysis provided in the previous section. One focuses on testing the accuracy in terms of the $L^2$-norm and $H^1$-norm, which are defined as follows
\begin{align}
L^2\text{-error} =  \|u(\boldsymbol{x},T)-u_h^N\|, \quad H^1\text{-error}=\|\nabla_h \left(u(\boldsymbol{x},T)-u_h^N\right)\|,\notag 
\end{align} 
where $\nabla_h\cdot$ is the piecewise defined gradient operator on mesh partitioning.
The other example is for verifying the decay property of the $L^2$-norm in the case where the right-hand side is zero and the robustness of the numerical scheme in curved-boundary domains. In the second example, we adopt a circular computational domain and employ unstructured polygonal meshes with varying polynomial degrees for testing purposes.
\begin{example}\label{example1}
In \eqref{202507262339}, we choose the parameters $\nu=\kappa=\alpha=\beta=\gamma=1$, $\Omega=(0,1)^2$ and the exact solution is taken as 
\begin{align}
u(\boldsymbol{x},t) = e^{it}\sin(x)\sin(y)(1-x)(1-y).
\end{align}
The right-hand side is solved from the equation based on the true solution above.

The numerical results corresponding to this example are tabulated in Tables \ref{table1} through \ref{table10}. The mesh partitions employed in this example are illustrated in Figures \ref{fig01} to \ref{fig3}. To assess the convergence accuracy in the spatial direction, we employ a sufficiently small time step and report the results for three categories of meshes and three polynomial degrees in Tables \ref{table1} to \ref{table9}. The convergence results for different polynomial degrees (\(k=1,2,3\)) on non-convex meshes with \(\theta=1/8\) are presented in Tables \ref{table1}--\ref{table3}. The orders of accuracy under Voronoi meshes (with $\theta=1/4$) are presented in Tables \ref{table4}--\ref{table6}, and the accuracy results under hybrid meshes (with $\theta=3/8$) are provided in Tables \ref{table7}--\ref{table9}. 
It can be seen from the above numerical results that the numerical schemes proposed in this paper have achieved the convergence orders predicted by the theoretical analysis in the previous section. Finally, the numerical results under non-convex meshes when $\theta=1/4$ and $k = 3$ in the temporal direction are presented in Table \ref{table10}.
It is in very good agreement with the theoretical second-order convergence accuracy.
\end{example}
\begin{figure}[h]
	\centering
	\includegraphics[width=7cm]{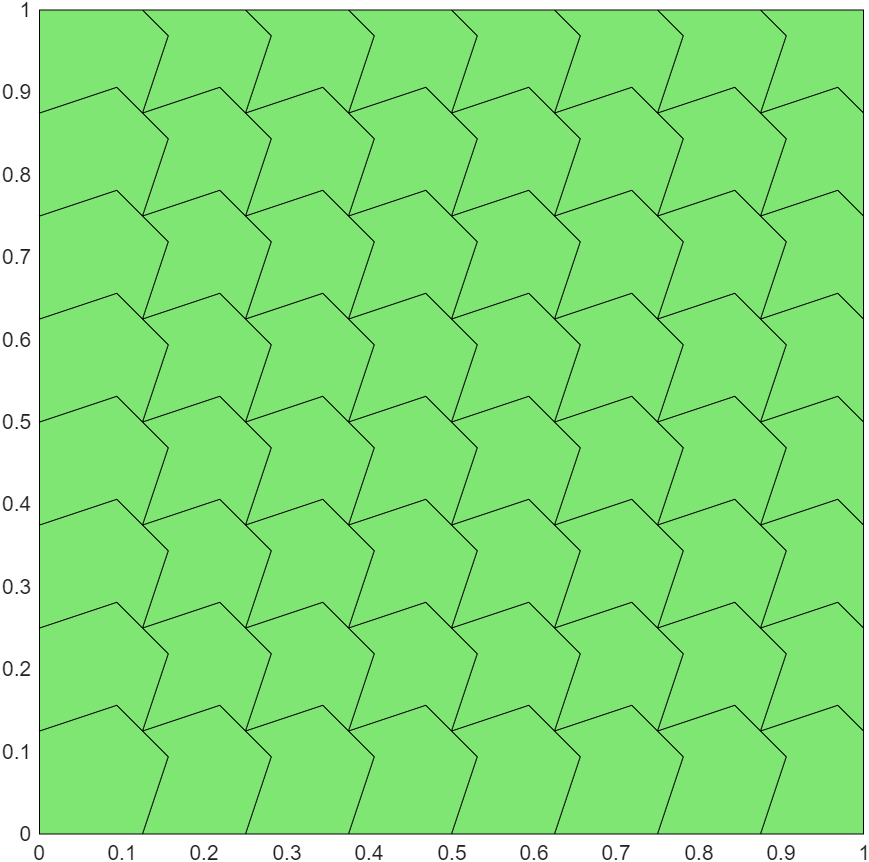}
	\caption{A non-convex partition composed of 64 elements.}\label{fig01}
\end{figure}

\begin{table}
	\centering 
	\caption{The convergence orders for Example \ref{example1} on non-convex meshes with $\theta=\frac{1}{8}$ and $k=1$.}
	\label{table1}
	\begin{tabular}{ccccc}
		\toprule 
		$h$ & $L^2$-error & $\text{Order}$ & $H^1$-error & $\text{Order}$\\
		\midrule
	1/2  & 6.77954e-03 &  --    & 8.57821e-02   & --\\
	1/4  & 1.96755e-03 &  1.7848 &  4.64516e-02  & 0.8849\\
	1/8  & 5.08362e-04  & 1.9525  & 2.36283e-02  & 0.9752\\
	1/16 & 1.27725e-04  & 1.9928  & 1.18364e-02  & 0.9973\\
	1/32 & 3.19239e-05  & 2.0003  & 5.90933e-03  & 1.0022\\
		\bottomrule
	\end{tabular}
\end{table}

\begin{table}\label{table2}
	\centering
	\caption{The convergence orders for Example \ref{example1} on non-convex meshes with $\theta=\frac{1}{8}$ and $k=2$.}
	\begin{tabular}{ccccc}
		\toprule 
		$h$ & $L^2$-error & $\text{Order}$ & $H^1$-error & $\text{Order}$\\
		\midrule
1/2 &  1.89377e-03 &  --      & 3.21908e-02  & --\\
1/4 &  2.59449e-04 &  2.8677  & 8.94403e-03  & 1.8477\\
1/8 &  3.21984e-05 &  3.0104  & 2.29592e-03  & 1.9619\\
1/16 &  3.95193e-06 &  3.0264  & 5.78684e-04  & 1.9882\\
1/32 &  4.87451e-07 &  3.0192  & 1.45110e-04  & 1.9956\\
		\bottomrule
	\end{tabular}
\end{table}

\begin{table}
	\centering
	\caption{The convergence orders for Example \ref{example1} on non-convex meshes with $\theta=\frac{1}{8}$ and $k=3$.}
	\label{table3}
	\begin{tabular}{ccccc}
		\toprule 
		$h$ & $L^2$-error & $\text{Order}$ & $H^1$-error & $\text{Order}$\\
		\midrule
1/2 &  3.32796e-04  & --  & 8.06780e-03 &  --\\
1/4  & 1.83062e-05   &4.1842  & 9.90653e-04  & 3.0257\\
1/8   &1.06468e-06   &4.1038  & 1.20608e-04  & 3.0381\\
1/16   &6.71119e-08   &3.9877  & 1.44979e-05  & 3.0564\\
		\bottomrule
	\end{tabular}
\end{table}

\begin{figure}[h]
	\centering
	\includegraphics[width=7cm]{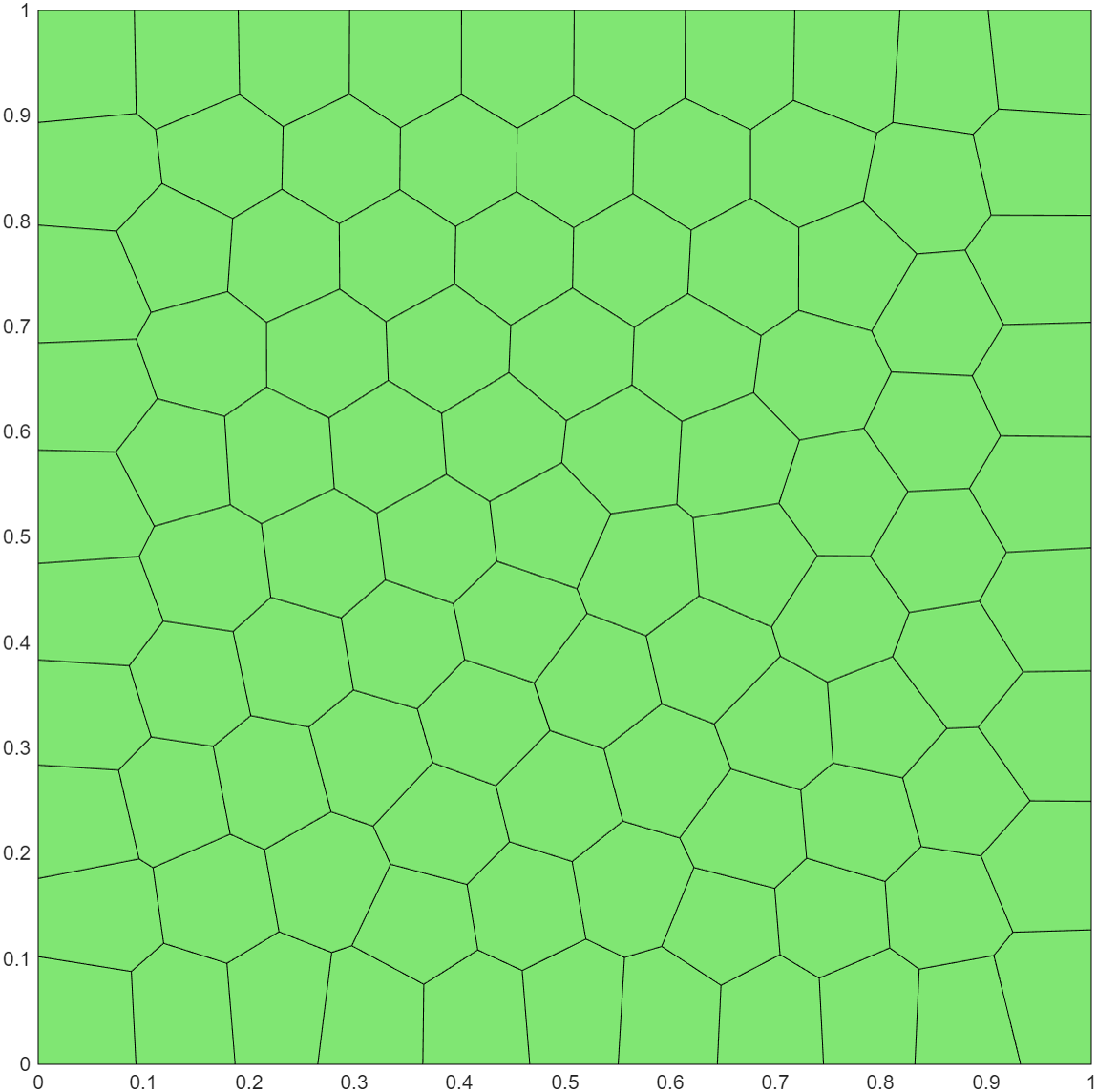}
	\caption{A Voronoi partition composed of 64 elements.}\label{fig2}
\end{figure}

\begin{table}
	\centering
	\caption{The convergence orders for Example \ref{example1} on Voronoi meshes with $\theta=\frac{1}{4}$ and $k=1$.}
	\label{table4}
	\begin{tabular}{ccccc}
		\toprule 
		$h$ & $L^2$-error & $\text{Order}$ & $H^1$-error & $\text{Order}$\\
		\midrule
	1/2 &  6.49409e-03 &  -- &  8.25620e-02  & --\\
	1/4 &  1.83974e-03 &  1.8196 &  4.45675e-02  & 0.8895\\
	1/8 &  4.42711e-04 &  2.0551 &  2.25716e-02  & 0.9815\\
	1/16 &  2.16718e-04 &  2.0611 &  1.58809e-02  & 1.0144\\
	1/32 &  2.61228e-05 &  2.0350 &  5.62040e-03  & 0.9990\\
		\bottomrule
	\end{tabular}
\end{table}

\begin{table}
	\centering
	\caption{The convergence orders for Example \ref{example1} on Voronoi meshes with $\theta=\frac{1}{4}$ and $k=2$.}\label{table5}
	\begin{tabular}{ccccc}
		\toprule 
		$h$ & $L^2$-error & $\text{Order}$ & $H^1$-error & $\text{Order}$\\
		\midrule
1/2 &  1.75795e-03 &  --  & 3.01168e-02  & --\\
1/4 &  2.43349e-04 &  2.8528  & 8.42798e-03  & 1.8373\\
1/8 &  2.76447e-05 &  3.1380  & 1.98436e-03  & 2.0865\\
1/16 &  9.30524e-06 &  3.1418  & 9.55196e-04  & 2.1096\\
1/32 &  3.93557e-07 &  3.0423  & 1.16588e-04  & 2.0229\\
		\bottomrule
	\end{tabular}
\end{table}

\begin{table}
	\centering
	\caption{The convergence orders for Example \ref{example1} on Voronoi meshes with $\theta=\frac{1}{4}$ and $k=3$.}
	\label{table6}
	\begin{tabular}{ccccc}
		\toprule 
		$h$ & $L^2$-error & $\text{Order}$ & $H^1$-error & $\text{Order}$\\
		\midrule
1/2 &  3.10919e-04 &  --  & 6.84783e-03 &  --\\
1/4 &  1.81557e-05 &  4.0980  & 8.53570e-04 &  3.0041\\
1/8 &  9.57198e-07 &  4.2455  & 9.58020e-05 &  3.1554\\
1/16 &  2.36222e-07 &  4.0373  & 3.28305e-05 &  3.0900\\
		\bottomrule
	\end{tabular}
\end{table}

\begin{figure}[h]
	\centering
	\includegraphics[width=7cm]{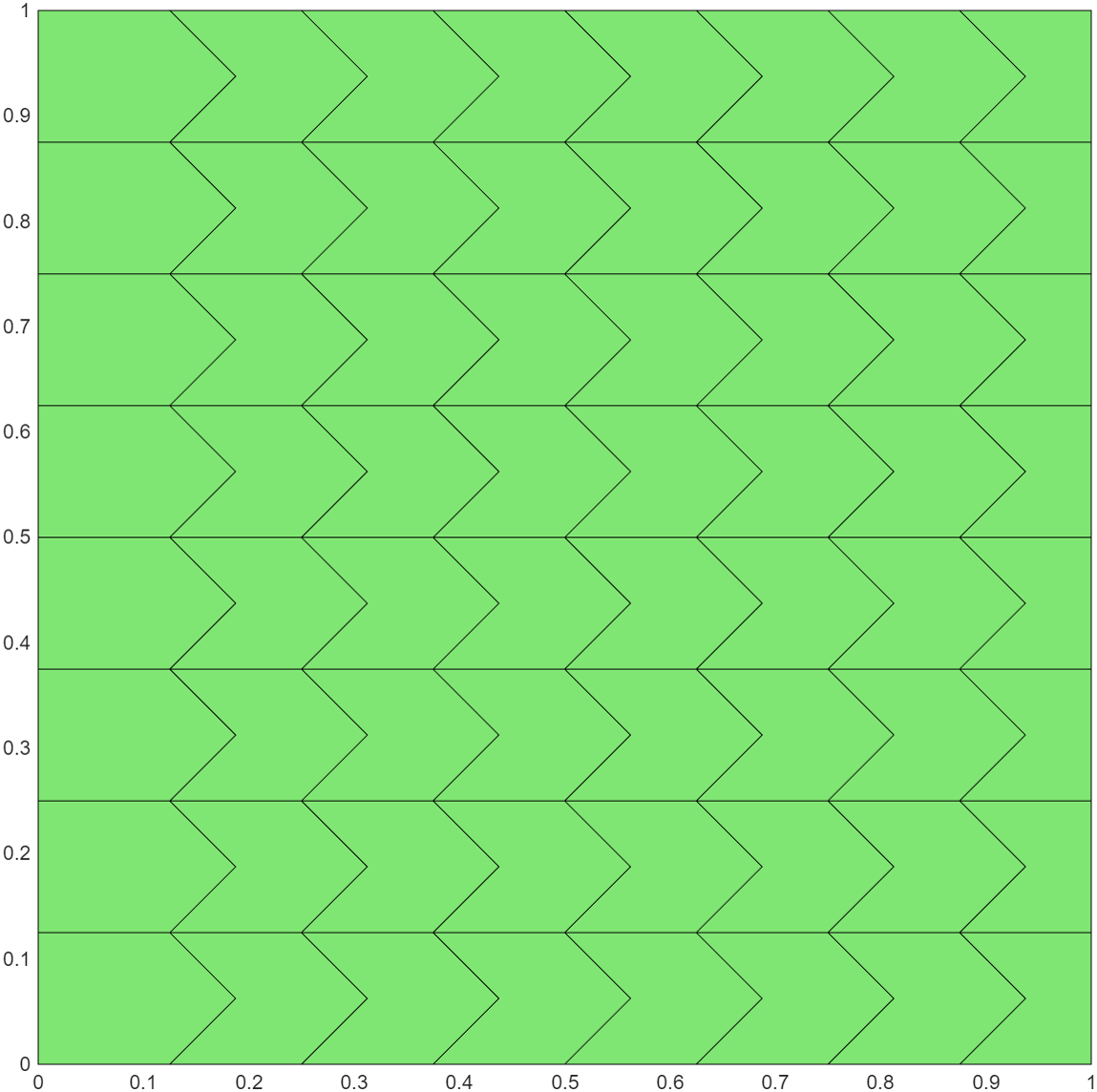}
	\caption{A mixed mesh partition containing concave and convex elements. }\label{fig3}
\end{figure}

\begin{table}
	\centering
	\caption{The convergence orders for Example \ref{example1} on mixed meshes with $\theta=\frac{3}{8}$ and $k=1$.}
	\label{table7}
	\begin{tabular}{ccccc}
		\toprule 
		$h$ & $L^2$-error & $\text{Order}$ & $H^1$-error & $\text{Order}$\\
		\midrule
	1/2  &  7.10731e-03 &  0.0000 &  8.77645e-02 &  0.0000\\
	1/4  &  2.07023e-03 &  1.7795 &  4.79082e-02 &  0.8734\\
	1/8  &  5.35665e-04 &  1.9504 &  2.45034e-02 &  0.9673\\
	1/16 &  1.34702e-04 &  1.9916 &  1.23081e-02 &  0.9934\\
	1/32 &  3.36800e-05 &  1.9998 &  6.16071e-03 &  0.9984\\
		\bottomrule
	\end{tabular}
\end{table}

\begin{table}\label{table8}
	\centering
	\caption{The convergence orders for Example \ref{example1} on mixed meshes with $\theta=\frac{3}{8}$ and $k=2$.}
	\begin{tabular}{ccccc}
		\toprule 
		$h$ & $L^2$-error & $\text{Order}$ & $H^1$-error & $\text{Order}$\\
		\midrule
	1/2 &  1.93394e-03  & --  & 3.38850e-02  & --\\
	1/4 &  2.66987e-04  & 2.8567  & 9.43686e-03  & 1.8443\\
	1/8 &  3.33308e-05  & 3.0018  & 2.41674e-03  & 1.9652\\
	1/16 &  4.10516e-06  & 3.0213  & 6.07827e-04  & 1.9913\\
	1/32 &  5.07309e-07  & 3.0165  & 1.52238e-04  & 1.9973\\
		\bottomrule
	\end{tabular}
\end{table}

\begin{table}
	\centering
	\caption{The convergence orders for Example \ref{example1} on mixed meshes with $\theta=\frac{3}{8}$ and $k=3$.}
	\label{table9}
	\begin{tabular}{ccccc}
		\toprule 
		$h$ & $L^2$-error & $\text{Order}$ & $H^1$-error & $\text{Order}$\\
		\midrule
	1/2&   3.41577e-04 &  --  & 8.39104e-03 &  --\\
	1/4&   1.88177e-05 &  4.1820  & 1.04558e-03 &  3.0045\\
	1/8&   1.08430e-06 &  4.1173  & 1.29440e-04 &  3.0140\\
	1/16&   6.59498e-08 &  4.0393  & 1.58202e-05 &  3.0324\\
		\bottomrule
	\end{tabular}
\end{table}

\begin{table}
	\centering
	\caption{The convergence orders in time direction for Example \ref{example1} on non-convex meshes with $\theta=\frac{1}{4}$ and $k=3$.}
	\label{table10}
	\begin{tabular}{ccccc}
		\toprule 
		$\tau$ & $L^2$-error & $\text{Order}$ & $H^1$-error & $\text{Order}$\\
		\midrule
	1/2 &  8.77078e-04 &  --   &3.92322e-03   &--\\
	1/4  & 1.61655e-04 &  2.4398  & 7.23625e-04  & 2.4387\\
	1/8  & 4.16780e-05 &  1.9556  & 1.86402e-04  & 1.9568\\
	1/16  & 1.04455e-05&   1.9964 &  4.67695e-05 &  1.9948\\
	1/32  & 2.61201e-06  & 1.9997 &  1.20221e-05 &  1.9599\\
		\bottomrule
	\end{tabular}
\end{table}

\begin{example}\label{example2}
	In \eqref{202507262339}, we set the parameters as  $\nu=\kappa=\alpha=\beta=\gamma=1$, take $\Omega=\{(x,y):x^2+y^2<1\}$ and define the source function $f$ according to the exact solution
	$$u(\boldsymbol{x},t)=i\sin\left(x^2+y^2-1\right)e^{-t}.$$
\end{example}

The purpose of this example is to test the robustness of the numerical algorithm in curved-edge domains and verify the boundedness of the $L^2$-norm. A polygonal partition of the unit circle is presented in Figure \ref{fig4}. It can be seen from Tables \ref{table11} and \ref{table12} that, in the circular domain, the computer implementation results of the fully discrete numerical scheme are consistent with the theoretical analysis. To test the boundedness of the numerical solution under the $L^2$-norm, we set the right-hand side term in the equation to zero. Obviously, there is no explicitly expressible solution in this case. Figure \ref{fig5} shows its variation trend over time, which is consistent with the argument result in Theorem \ref{theorem1}. This further demonstrates the robustness of the numerical scheme in this paper.

\begin{figure}[h]
	\centering
	\includegraphics[width=7cm]{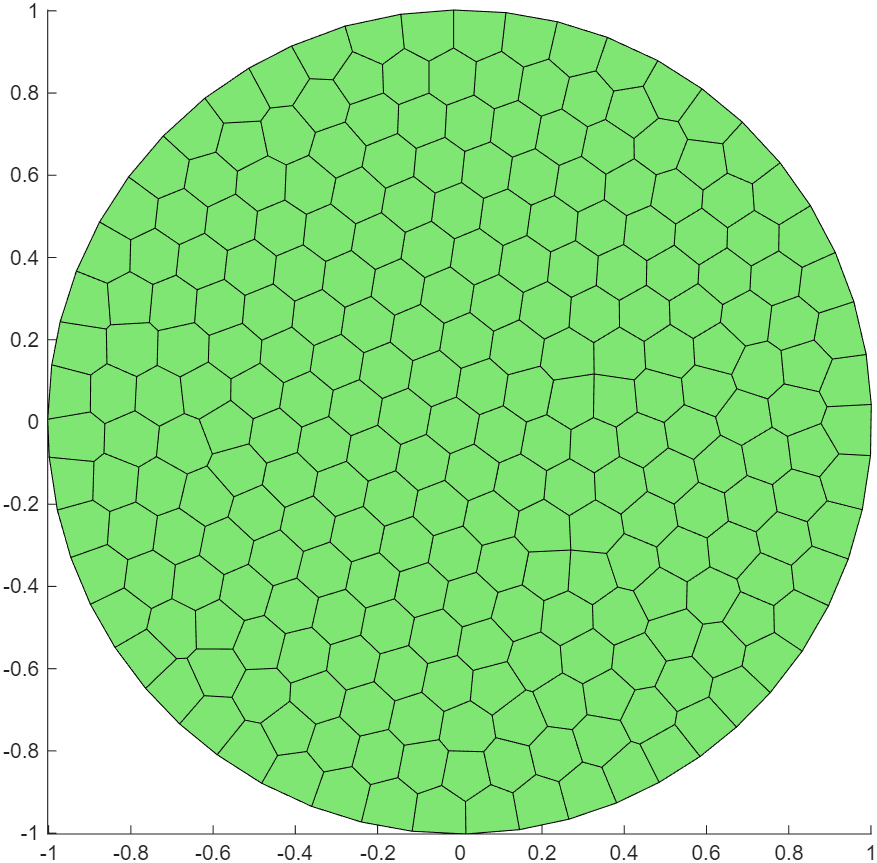}
	\caption{A Voronoi mesh partition for the unit circle.}\label{fig4}
\end{figure}

\begin{table}
	\centering
	\caption{The convergence orders for Example \ref{example2} on Voronoi meshes with $\theta=\frac{1}{4}$ and $k=1$.}
	\label{table11}
	\begin{tabular}{ccccc}
		\toprule 
		$h$ & $L^2$-error & $\text{Order}$ & $H^1$-error & $\text{Order}$\\
		\midrule
1/4  & 4.28568e-02 &  --  & 6.90124e-01 &  --\\
1/8  &9.33613e-03  & 2.1986   &3.30723e-01  & 1.0612\\
1/16 &  2.25235e-03 &  2.0514  & 1.63946e-01 &  1.0124\\
1/32 &  5.53500e-04 &  2.0248   &8.15875e-02 &  1.0068\\
		\bottomrule
	\end{tabular}
\end{table}

\begin{table}
	\centering
	\caption{The convergence orders for Example \ref{example2} on Voronoi meshes with $\theta=\frac{1}{4}$ and $k=2$.}
	\label{table12}
	\begin{tabular}{ccccc}
		\toprule 
		$h$ & $L^2$-error & $\text{Order}$ & $H^1$-error & $\text{Order}$\\
		\midrule
1/4  & 1.77305e-03 &  --  & 4.40879e-02 &  --\\
1/8   &2.55595e-04  & 2.7943  & 1.30722e-02  & 1.7539\\
1/16   &3.36348e-05  & 2.9258 & 3.45896e-03  & 1.9181\\
1/32   &4.45948e-06   &2.9150 &  8.90269e-04 &  1.9580\\
		\bottomrule
	\end{tabular}
\end{table}

\begin{figure}[h]
	\centering
	\includegraphics[width=10cm]{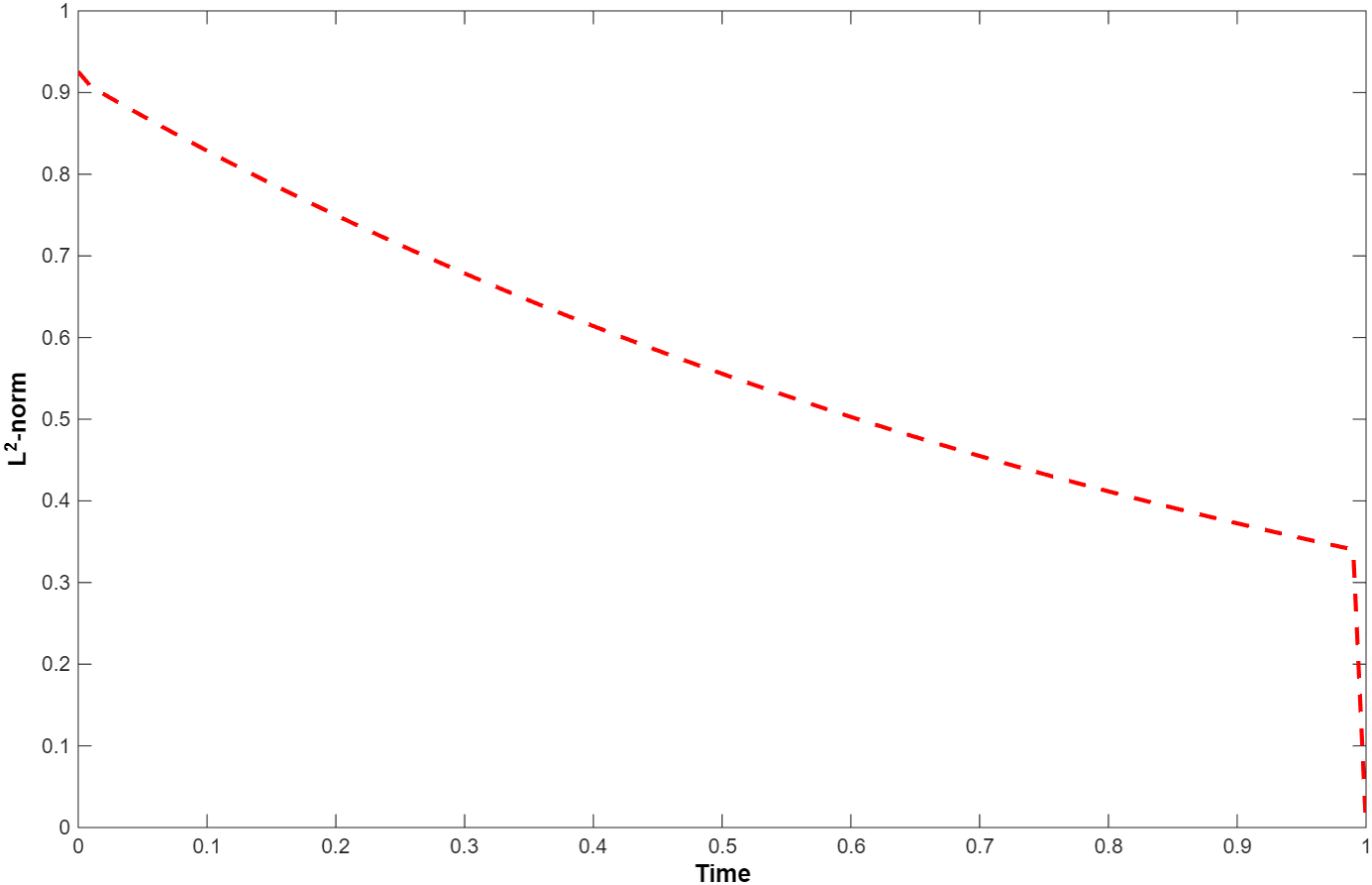}
	\caption{$L^2$-$\text{norm}$ of the numerical solution for Example \ref{example2} with $h=1/30$ and $\tau=1/100$.}\label{fig5}
\end{figure}

\section{Conclusions}\label{section5}
In this paper, we propose a novel linearized fully discrete discontinuous finite element scheme, which adopts the second-order $\theta$ scheme and the polygonal discontinuous finite element method in the temporal and spatial directions, respectively. The stability of the numerical method and its unconditional convergence under the $L^2$-norm are rigorously proven by means of the generalized inverse inequality and the transfer formula. There are the following two issues worthy of further consideration in the future. Firstly, high-order discrete methods can be adopted in the temporal direction, such as the Gaussian collocation method and the discontinuous finite element method. Secondly, the adaptive polygonal discontinuous Galerkin method deserves further exploration in the future.

\section*{Declaration of competing interest}
The authors declare that they have no known competing financial interests or personal relationships that could have appeared to influence the work reported in this paper.

\section*{Data availability}
Data will be made available on request.

\section*{Acknowledgments}
This work is supported by the Doctoral Starting Foundation of Pingdingshan University (No. PXY-BSQD2023022) and the Natural Science Foundation of Henan Province (Nos. 242300420655, 252300420343).

\end{document}